\documentclass[12pt,reqno]{amsart}
\usepackage{amsmath,amssymb}

\calclayout
\theoremstyle{plain}
\newtheorem{thm}{Theorem}[section]
\newtheorem{theorem}[thm]{Theorem}

\newtheorem{proposition}[thm]{Proposition}
\newtheorem{corollary}[thm]{Corollary}
\theoremstyle{definition}
\newtheorem{definition}[thm]{Definition}

\newtheorem{remark}[thm]{Remark}

\allowdisplaybreaks[1]

\newcommand{\bb}{\boldsymbol}
\newcommand{\wt}{\mathrm{wt}}
\newcommand{\dep}{\mathrm{dep}}
\newcommand{\calA}{\mathcal{A}}

\title{Ohno-type identities for multiple harmonic sums}
\author{Shin-ichiro Seki and Shuji Yamamoto}
\address{Mathematical Institute, Tohoku University, 6-3,Aoba, Aramaki, Aoba-Ku, 
Sendai 980-8578, Japan}
\email{shinichiro.seki.b3@tohoku.ac.jp}
\address{Keio Institute of Pure and Applied Sciences (KiPAS), 
Graduate School of Science and Technology, Keio University, 3-14-1 Hiyoshi, 
Kohoku-ku, Yokohama, 223-8522, Japan}
\email{yamashu@math.keio.ac.jp}
\thanks{This work was supported in part by JSPS KAKENHI Grant Numbers 
JP18J00151, JP16H06336, JP16K13742, JP18K03221, 
as well as the KiPAS program 2013--2018 of the Faculty of Science and Technology 
at Keio University. }

\begin{document}

\begin{abstract}
We establish Ohno-type identities for multiple harmonic ($q$-)sums 
which generalize Hoffman's identity and Bradley's identity. 
Our result leads to a new proof of the Ohno-type relation for 
$\mathcal{A}$-finite multiple zeta values recently proved 
by Hirose, Imatomi, Murahara and Saito. As a further application, 
we give certain sum formulas for $\calA_2$- or $\calA_3$-finite multiple zeta values. 
\end{abstract}

\maketitle

\section{Introduction}
Let $N$ be a positive integer. Euler \cite{E} proved the following identity for 
the $N$-th harmonic number: 
\begin{equation}\label{eq:Euler}
\sum_{m=1}^N\frac{(-1)^{m-1}}{m}\binom{N}{m}=\sum_{n=1}^N\frac{1}{n}.
\end{equation}
It is known today that there are various generalizations of Euler's identity. 
We call a tuple of positive integers an index. 
For an index $\bb{k}=(k_1, \dots, k_r)$,  we write it in the form
\[
\bb{k}=(\{1\}^{a_1-1}, b_1+1, \dots, \{1\}^{a_{s-1}-1}, b_{s-1}+1, \{1\}^{a_s-1}, b_s),
\]
where $a_1, \dots, a_s, b_1, \dots, b_s$ are positive integers and $\{1\}^a$ means $1, \dots, 1$ repeated $a$ times, and then we define its Hoffman dual $\bb{k}^{\vee}$ by 
\[
\bb{k}^{\vee}:=(a_1, \{1\}^{b_1-1}, a_2+1, \{1\}^{b_2-1}, \dots, a_s+1, \{1\}^{b_s-1}). 
\]
Let $\bb{k}=(k_1, \dots, k_r)$ and $\bb{k}^{\vee}=(l_1, \dots, l_s)$. 
After Roman \cite{Rom} (the case $r=1$) and Hernandez \cite{BDWLH} (the case $s=1$), 
Hoffman \cite{H} proved 
\begin{equation}\label{eq:Hoffman}
\sum_{1\leq m_1\leq\cdots\leq m_r\leq N}
\frac{(-1)^{m_r-1}}{m_1^{k_1}\cdots m_r^{k_r}}\binom{N}{m_r}
=\sum_{1\leq n_1\leq\cdots\leq n_s\leq N}
\frac{1}{n_1^{l_1}\cdots n_s^{l_s}}. 
\end{equation}

There are also $q$-analogs of these identities. 
Let $q$ be a real number satisfying $0<q<1$. 
For an integer $m$, we define the $q$-integer $[m]_q:=\frac{1-q^m}{1-q}$. 
When $0\leq m\leq N$, we define the $q$-factorial $[m]_q!:=\prod_{a=1}^m[a]_q$ 
($[0]_q:=1$) and the $q$-binomial coefficient 
$\binom{N}{m}_q:=\frac{[N]_q!}{[m]_q![N-m]_q!}$. 
Van Hamme \cite{V} proved a $q$-analog of Euler's identity \eqref{eq:Euler}
\[
\sum_{m=1}^N\frac{(-1)^{m-1}q^{\frac{m(m+1)}{2}}}{[m]_q}\binom{N}{m}_q
=\sum_{n=1}^N\frac{q^n}{[n]_q}. 
\]
After Dilcher \cite{D} (the case $r=1$) and Prodinger \cite{P} (the case $s=1$), 
Bradley \cite{B2} proved a $q$-analog of Hoffman's identity (\ref{eq:Hoffman})
\begin{equation}\label{eq:Bradley} %multline?
\begin{split}
\sum_{1\leq m_1\leq\cdots\leq m_r\leq N}
\frac{q^{(k_1-1)m_1+\cdots+(k_r-1)m_r}}{[m_1]_q^{k_1}\cdots [m_r]_q^{k_r}}
&\cdot(-1)^{m_r-1}q^{\frac{m_r(m_r+1)}{2}}\binom{N}{m_r}_q\\
&=\sum_{1\leq n_1\leq\cdots\leq n_s\leq N}
\frac{q^{n_1+\cdots+n_s}}{[n_1]_q^{l_1}\cdots[n_s]_q^{l_s}}. 
\end{split}
\end{equation}

The equality \eqref{eq:Hoffman} or \eqref{eq:Bradley} is a kind of duality 
for multiple harmonic ($q$-)sums. 
Since the duality relations for ($q$-)multiple zeta values are generalized 
to Ohno's relations (\cite{O,B1}), it is natural to ask whether (and how) 
we can generalize \eqref{eq:Hoffman} and \eqref{eq:Bradley} to Ohno-type identities. 
This question was considered by Oyama \cite{Oyama} 
and more recently by Hirose, Imatomi, Murahara and Saito \cite{HIMS}. 
More precisely, they treated identities of the $\mathcal{A}$-finite multiple zeta values, 
that is, congruences modulo prime numbers. 

In this article, we prove Ohno-type identities which generalize \eqref{eq:Bradley} 
(Theorem \ref{MT1}) and \eqref{eq:Hoffman} (Corollary \ref{cor:Ohno}). 
We stress that our formulas are true identities, not congruences. 
This allows us to give, besides a new proof of 
Hirose-Imatomi-Murahara-Saito's relation for $\mathcal{A}$-finite multiple zeta values, 
sum formulas for $\calA_2$- or $\calA_3$-finite multiple zeta values, 
which are congruences modulo square or cube of primes. 

\section{Main Results}
\subsection{Ohno-type identity}
For a tuple of non-negative integers $\bb{e}=(e_1, \dots, e_r)$, 
we define its weight $\wt(\bb{e})$ and depth $\dep(\bb{e})$ 
to be $e_1+\cdots+e_r$ and $r$, respectively. 
Let $J_{e,r}$ be the set of all tuples of non-negative integers $\bb{e}$ 
such that $\wt(\bb{e})=e$, $\dep(\bb{e})=r$, and 
set $J_{\ast,r}:=\bigcup_{e=0}^{\infty}J_{e,r}$. 
For $\bb{e}_1,\bb{e}_2\in J_{\ast,r}$, $\bb{e}_1+\bb{e}_2$ denotes the entrywise sum. 
Similarly, let $I_{k,r}$ be the set of all indices $\bb{k}$ such that $\wt(\bb{k})=k$, 
$\dep(\bb{k})=r$, and set $I_{\ast,r}:=\bigcup_{k=0}^{\infty}I_{k,r}$. 
By convention, $I_{\ast,0}=\{\varnothing\}$ is the set consisting only of the empty index. 

For $\bb{k}=(k_1,\dots,k_r)\in I_{\ast,r}$ and $\bb{e}=(e_1,\dots,e_r) \in J_{\ast,r}$, put 
\[
b(\bb{k};\bb{e}):=\prod_{i=1}^r\binom{k_i+e_i+\delta_{i1}+\delta_{ir}-2}{e_i},
\]
where $\delta_{ij}$ is Kronecker's delta. Here, we use the convention that
\[
\binom{e-1}{e}=\begin{cases}1 & (e=0), \\ 0 &(e>0). \end{cases}
\]
For a positive integer $N$, $\bb{k}=(k_1,\dots,k_r) \in I_{\ast, r}$ and 
$\bb{e}=(e_1,\dots,e_r) \in J_{\ast,r}$, 
we define the multiple harmonic $q$-sums 
$H_N^{\star}(\bb{k};q)$ and $z_N^{\star}(\bb{k};\bb{e};q)$ by
\begin{align*}
H_N^{\star}(\bb{k};q)&:=\sum_{1 \leq m_1\leq\cdots\leq m_r\leq N}
\frac{q^{(k_1-1)m_1+\cdots+(k_r-1)m_r}}{[m_1]_q^{k_1}\cdots[m_r]_q^{k_r}}
\cdot(-1)^{m_r-1}q^{\frac{m_r(m_r+1)}{2}}\binom{N}{m_r}_q,\\
z_N^{\star}(\bb{k};\bb{e};q)&:=\sum_{1\leq m_1\leq\cdots\leq m_r\leq N}
\frac{q^{(e_1+1)m_1+\cdots+(e_r+1)m_r}}{[m_1]_q^{k_1+e_1}\cdots[m_r]_q^{k_r+e_r}}.
\end{align*}
We set $z_N^{\star}(\bb{k};q):=z_N^{\star}(\bb{k};\{0\}^r;q)$ and 
$z_N^{\star}(\varnothing;q):=1$. 
%Here $\{a\}^r$ means $r$ times repetition $\underbrace{a, \dots, a}_r$ of $a$. 
The first main result is the following:

\begin{theorem}\label{MT1}
Let $N$ be a positive integer, $e$ a non-negative integer and 
$\bb{k}\in I_{\ast,r}$ an index. 
Set $s:=\dep(\bb{k}^{\vee})$. Then we have
\begin{equation}\label{q-Ohno}
\sum_{\bb{e} \in J_{e,r}}b(\bb{k};\bb{e})H_N^{\star}(\bb{k}+\bb{e};q)
=\sum_{j=0}^ez_N^{\star}(\{1\}^{e-j};q)
\sum_{\bb{e'} \in J_{j,s}}z_N^{\star}(\bb{k}^{\vee};\bb{e'};q).
\end{equation}
\end{theorem}

The case $e=0$ gives Bradley's identity $H_N^{\star}(\bb{k};q)=z_N^{\star}(\bb{k}^{\vee};q)$. 
We will prove \eqref{q-Ohno} by using a certain \emph{connected sum} 
in \S\ref{sec:connected sum}, based on the same idea used in another paper of 
the authors \cite{SY}. 
This proof is new even if one specializes it to Hoffman's identity. 

Let 
\begin{align}
H_N^{\star}(\bb{k})&:=\lim_{q\to 1}H_N^{\star}(\bb{k};q)
=\sum_{1 \leq m_1\leq\cdots\leq m_r\leq N}
\frac{(-1)^{m_r-1}}{m_1^{k_1}\cdots m_r^{k_r}}\binom{N}{m_r}, \notag\\
\zeta_N^\star(\bb{k})&:=\lim_{q\to 1}z_N^{\star}(\bb{k};q)
=\sum_{1\leq m_1\leq\cdots\leq m_r\leq N}
\frac{1}{m_1^{k_1}\cdots m_r^{k_r}}. \label{eq:zeta^star_N}
\end{align}
By taking the limit $q\to 1$ in \eqref{q-Ohno}, we obtain the following:

\begin{corollary}\label{cor:Ohno}
Let $N$ be a positive integer, $e$ a non-negative integer and 
$\bb{k}\in I_{\ast,r}$ an index. 
Set $s:=\dep(\bb{k}^{\vee})$. Then we have
\begin{equation}\label{Ohno}
\sum_{\bb{e} \in J_{e,r}}b(\bb{k};\bb{e})H_N^{\star}(\bb{k}+\bb{e})
=\sum_{j=0}^e\zeta_N^{\star}(\{1\}^{e-j})
\sum_{\bb{e'} \in J_{j,s}}\zeta_N^{\star}(\bb{k}^{\vee}+\bb{e'}).
\end{equation}
%\begin{equation}
%\sum_{i=0}^l(-1)^{l-i}\zeta_N(\{1\}^{l-i})\sum_{\bb{e} \in J_{i,r}}b(\bb{k};\bb{e})H_N^{\star}(\bb{k}+\bb{e})=\sum_{\bb{e}' \in J_{l,s}}\zeta_N^{\star}(\bb{k}^{\vee}+\bb{e'})
%\end{equation}
\end{corollary}
The case $e=0$ gives Hoffman's identity 
$H_N^{\star}(\bb{k})=\zeta_N^{\star}(\bb{k}^{\vee})$. 

For an application of \eqref{Ohno}, we recall $\calA$-finite multiple zeta values. 
First we define a $\mathbb{Q}$-algebra $\calA$ by
\[
\calA:=\Biggl(\prod_{p\colon\text{prime}}\mathbb{Z}/p\mathbb{Z} \Biggr) 
\Biggm/ \Biggl(\bigoplus_{p\colon\text{prime}}\mathbb{Z}/p\mathbb{Z}\Biggr).
\]
For a positive integer $N$ and an index $\bb{k}=(k_1,\dots,k_r)\in I_{\ast,r}$, 
we define a multiple harmonic sum $\zeta_N(\bb{k})$ by
\[
\zeta_N(\bb{k}):=\sum_{1\leq m_1<\cdots<m_r\leq N}
\frac{1}{m_1^{k_1}\cdots m_r^{k_r}}
\]
(compare with $\zeta_N^\star(\bb{k})$ given in \eqref{eq:zeta^star_N}). 
We set $\zeta_N(\varnothing)=\zeta_N^{\star}(\varnothing)=1$ by convention. 
Then the $\calA$-finite multiple zeta values $\zeta_{\calA}(\bb{k})$ and
$\zeta_{\calA}^{\star}(\bb{k})$ are defined by
\[
\zeta_{\calA}(\bb{k}):=\bigl(\zeta_{p-1}(\bb{k})\bmod p\bigr)_p, \quad 
\zeta_{\calA}^{\star}(\bb{k}):=\bigl(\zeta_{p-1}^{\star}(\bb{k})\bmod p\bigr)_p \in \calA. 
\]

Since $(-1)^{m-1}\binom{p-1}{m} \equiv -1 \pmod{p}$ holds for any prime $p$ 
greater than $m$, we have 
\[\bigl(H_{p-1}^\star(\bb{k})\bmod p\bigr)_p=-\zeta_\calA^\star(\bb{k}). \]
Moreover, it is known that $\zeta_\calA^\star(\{1\}^e)=0$ for $e>0$, 
while $\zeta_\calA^\star(\varnothing)=1$. % What to refer?  
Hence we obtain the following relation among 
$\calA$-finite multiple zeta values as a corollary of \eqref{Ohno}. 

\begin{corollary}[Hirose-Imatomi-Murahara-Saito \cite{HIMS}]
Let $e$ be a non-negative integer and $\bb{k} \in I_{\ast,r}$ an index. 
Set $s:=\dep(\bb{k}^{\vee})$. Then we have
\[
\sum_{\bb{e} \in J_{e,r}}b(\bb{k};\bb{e})\zeta_{\calA}^{\star}(\bb{k}+\bb{e})
=-\sum_{e' \in J_{e,s}}\zeta_{\calA}^{\star}(\bb{k}^{\vee}+\bb{e'}).
\]
\end{corollary}

\subsection{Sum formulas for finite multiple zeta values}
Before stating our second main result, 
let us recall the sum formulas for $\calA$-finite multiple zeta values. 
First, it is easily seen that 
\begin{equation}\label{eq:k,r sum A}
\sum_{\bb{k}\in I_{k,r}}\zeta_\calA(\bb{k})
=\sum_{\bb{k}\in I_{k,r}}\zeta_\calA^\star(\bb{k})=0, 
\end{equation}
but this is not an analog of the sum formula for the multiple zeta values \cite{G}, 
since the admissibility condition $k_r\geq 2$ is ignored in \eqref{eq:k,r sum A}. 
A more precise analog (and its generalization) is due to Saito-Wakabayashi \cite{SW}. 
For integers $k, r$ and $i$ satisfying $1\leq i\leq r<k$, 
we put $I_{k,r,i}:=\{(k_1, \dots, k_r) \in I_{k,r}\mid k_i \geq 2\}$ and 
$B_{\bb{p}-k}:=(B_{p-k}\bmod{p})_p \in \calA$, 
where $B_n$ denotes the $n$-th Seki-Bernoulli number. 
Note that $B_{\bb{p}-k}=0$ if $k$ is even. 

\begin{theorem}[Saito-Wakabayashi \cite{SW}]\label{SW-thm}
Let $k, r$ and $i$ be integers satisfying $1\leq i\leq r<k$. 
Then, in the ring $\calA$, we have equalities 
\begin{align*}
\sum_{\bb{k} \in I_{k, r, i}}\zeta_{\calA}(\bb{k})
&=(-1)^{i}\biggl\{\binom{k-1}{i-1}+(-1)^r\binom{k-1}{r-i}\biggr\}\frac{B_{\bb{p}-k}}{k},\\
\sum_{\bb{k} \in I_{k, r, i}}\zeta_{\calA}^{\star}(\bb{k})
&=(-1)^{i}\biggl\{\binom{k-1}{r-i}+(-1)^r\binom{k-1}{i-1}\biggr\}\frac{B_{\bb{p}-k}}{k}. 
\end{align*}
\end{theorem}
In particular, if $k$ is even, we see that 
\begin{equation}\label{eq:SW even}
\sum_{\bb{k} \in I_{k,r,i}}\zeta_{\calA}(\bb{k})
=\sum_{\bb{k} \in I_{k,r,i}}\zeta_{\calA}^{\star}(\bb{k})=0. 
\end{equation}
Our aim is to lift the identities \eqref{eq:k,r sum A} and \eqref{eq:SW even} in $\calA$, 
which represent systems of congruences modulo (almost all) primes $p$, 
to congruences modulo $p^2$ or $p^3$, by using the identity \eqref{Ohno}. 

Let $n$ be a positive integer. In accordance with \cite{Ros,S,Z2}, we define a $\mathbb{Q}$-algebra $\calA_n$ by
\[
\calA_n:=\Biggl(\prod_{p\colon \text{prime}}\mathbb{Z}/p^n\mathbb{Z}\Biggr) 
\Biggm/ \Biggl(\bigoplus_{p\colon \text{prime}}\mathbb{Z}/p^n\mathbb{Z}\Biggr)
\]
and the $\calA_n$-finite multiple zeta values 
$\zeta_{\calA_n}(\bb{k})$ and $\zeta_{\calA_n}^{\star}(\bb{k})$ by
\[
\zeta_{\calA_n}(\bb{k}):=(\zeta_{p-1}(\bb{k})\bmod{p^n})_p, \quad 
\zeta_{\calA_n}^{\star}(\bb{k}):=(\zeta_{p-1}^{\star}(\bb{k})\bmod{p^n})_p \in \calA_n.
\]
We use the symbol $B_{\bb{p}-k}$ again to denote 
the element $(B_{p-k} \bmod{p^n})_p$ of $\calA_n$, 
and put $\bb{p}:=(p\bmod{p^n})_p \in \calA_n$. 
Then our second main result is the following: 

\begin{theorem}[= Proposition \ref{S_A_2,k,r} + Theorem \ref{thm:A_3-sum} + Theorem \ref{thm:k,r,i sum A_2}]
\label{MT2}
Let $k, r$ be positive integers satisfying $r\leq k$. Then, in the ring $\mathcal{A}_2$, we have
\[
\sum_{\bb{k} \in I_{k, r}}\zeta_{\calA_2}(\bb{k})
=(-1)^{r-1}\binom{k}{r}\frac{B_{\bb{p}-k-1}}{k+1}\bb{p},\quad 
\sum_{\bb{k} \in I_{k, r}}\zeta_{\calA_2}^{\star}(\bb{k})
=\binom{k}{r}\frac{B_{\bb{p}-k-1}}{k+1}\bb{p}.
\]
If $k$ is odd, in the ring $\calA_3$, we have
\[
\sum_{\bb{k} \in I_{k, r}}\zeta_{\calA_3}(\bb{k})
=(-1)^r\frac{k+1}{2}\binom{k}{r}\frac{B_{\bb{p}-k-2}}{k+2}\bb{p}^2,\quad 
\sum_{\bb{k} \in I_{k, r}}\zeta_{\calA_3}^{\star}(\bb{k})
=-\frac{k+1}{2}\binom{k}{r}\frac{B_{\bb{p}-k-2}}{k+2}\bb{p}^2.
\]
Furthermore, let $i$ be an integer satisfying $1\leq i\leq r$ and 
we assume that $k$ is even and greater than $r$. Then the equalities 
\[
\sum_{\bb{k} \in I_{k, r,i}}\zeta_{\calA_2}(\bb{k})
=(-1)^{r-1}\frac{a_{k,r,i}}{2}\cdot\frac{B_{\bb{p}-k-1}}{k+1}\bb{p},\quad 
\sum_{\bb{k} \in I_{k, r, i}}\zeta_{\calA_2}^{\star}(\bb{k})
=\frac{b_{k,r,i}}{2}\cdot\frac{B_{\bb{p}-k-1}}{k+1}\bb{p}
\]
hold in $\calA_2$. Here the coefficients $a_{k,r,i}$ and $b_{k,r,i}$ are given by 
\begin{align*}
a_{k,r,i}&:=\binom{k-1}{r}
+(-1)^{r-i}\biggl\{(k-r)\binom{k}{i-1}+\binom{k-1}{i-1}+(-1)^{r-1}\binom{k-1}{r-i}\biggr\},\\
b_{k,r,i}&:=\binom{k-1}{r}+
(-1)^{i-1}\biggl\{(k-r)\binom{k}{r-i}+\binom{k-1}{r-i}+(-1)^{r-1}\binom{k-1}{i-1}\biggr\}.
\end{align*}
\end{theorem}
We will prove this theorem in \S\ref{sec:sum formula} and \S\ref{sec:A_3-sum formula}. 

\section{The proof of Theorem \ref{MT1}}\label{sec:connected sum}
\begin{definition}[Connected sum]
Let $N$ be a positive integer, $q$ a real number satisfying $0<q<1$ and 
$x$ an indeterminate. Let $r>0$ and $s\geq0$ be integers. 
For $\bb{k}=(k_1,\dots,k_r) \in J_{\ast, r}$ satisfying $k_1, \dots, k_{r-1} \geq 1$ 
and $\bb{l}=(l_1, \dots, l_s) \in I_{\ast, s}$, 
we define a formal power series $Z_N^{\star}(\bb{k};\bb{l};q;x)$ in $x$ by
\[
Z_N^{\star}(\bb{k};\bb{l};q;x)
:=\sum_{1\leq m_1\leq\cdots\leq m_r\leq n_1\leq\cdots\leq n_s\leq n_{s+1}=N}
F_1(\bb{k};\bb{m};q;x)C(m_r,n_1,q,x)F_2(\bb{l};\bb{n};q;x),
\]
where
\begin{align*}
F_1(\bb{k};\bb{m};q;x)
&:=\frac{[m_1]_q}{[m_1]_q-q^{m_1}x}
\prod_{i=1}^r\frac{q^{(k_i-1)m_i}}{[m_i]_q([m_i]_q-q^{m_i}x)^{k_i-1}}
\cdot\frac{[m_r]_q}{[m_r]_q-q^{m_r}x},\\
C(m_r,n_1,q,x)
&:=(-1)^{m_r-1}q^{\frac{m_r(m_r+1)}{2}}
\frac{\prod_{h=1}^{n_1}([h]_q-q^hx)}{[m_r]_q![n_1-m_r]_q!},\\
F_2(\bb{l};\bb{n};q;x)
&:=\prod_{j=1}^s\frac{q^{n_j}}{([n_j]_q-q^{n_j}x)[n_j]_q^{l_j-1}}
\end{align*}
for $\bb{m}=(m_1,\dots, m_r)$ and $\bb{n}=(n_1,\dots, n_s)$.
\end{definition}

\begin{remark}
The sum $Z_N^{\star}(\bb{k};\bb{l};q;x)$ consists of two parts 
\[\sum_{1\leq m_1\leq\cdots\leq m_r\leq N} F_1(\bb{k};\bb{m};q;x) 
\ \text{ and }\ \sum_{1\leq n_1\leq\cdots\leq n_s\leq N} F_2(\bb{l};\bb{n};q;x), \]
connected by the factor $C(m_r,n_1,q,x)$ (and the relation $m_r\leq n_1$). 
We call it a connected sum with connector $C(m_r,n_1,q,x)$. 
In \cite{SY}, another type of connected sums is used 
to give a new proof of Ohno's relation for the multiple zeta values 
and Bradley's $q$-analog of it. 
\end{remark}

\begin{theorem}\label{step-thm}
For $(k_1, \dots, k_r) \in J_{\ast, r}$ with $k_1, \dots, k_{r-1} \geq 1$ 
and $(l_1, \dots, l_s) \in I_{\ast, s}$, we have
\begin{equation}\label{+1;1,}
Z_N^{\star}(k_1, \dots, k_r+1;l_1, \dots, l_s; q;x)
=Z_N^{\star}(k_1, \dots, k_r;1,l_1, \dots, l_s; q;x). 
\end{equation}
Moreover, if $s>0$, we also have 
\begin{equation}\label{+1,0;1+}
Z_N^{\star}(k_1, \dots, k_r+1,0;l_1, \dots, l_s; q;x)
=Z_N^{\star}(k_1, \dots, k_r;1+l_1, \dots, l_s; q;x).
\end{equation}
\end{theorem}

\begin{proof}
The equality \eqref{+1;1,} follows from the telescoping sum
\begin{align*}
&\frac{q^m}{[m]_q-q^mx}\cdot C(m,n,q,x)\\
&=\sum_{a=m}^n\biggl(\frac{q^m}{[m]_q-q^mx}\cdot C(m,a,q,x)
-\frac{q^m}{[m]_q-q^mx}\cdot C(m,a-1,q,x)\biggr)\\
&=\sum_{a=m}^nC(m,a,q,x)\cdot\frac{q^a}{[a]_q-q^ax}
\end{align*}
applied to $m=m_r$, $n=n_2$ and $a=n_1$ 
in the definition of $Z_N^{\star}(k_1, \dots, k_r;1,l_1, \dots, l_s;q;x)$. 
Similarly, the equality \eqref{+1,0;1+} follows from the telescoping sum
\begin{align*}
&\frac{q^m}{[m]_q}\sum_{a=m}^nq^{-a}C(a,n,q,x)\\
&=\frac{q^m}{[m]_q}\sum_{a=m}^n\biggl(\frac{[a]_q}{q^a}\cdot C(a,n,q,x)\cdot\frac{1}{[n]_q}
-\frac{[a+1]_q}{q^{a+1}}\cdot C(a+1,n,q,x)\cdot\frac{1}{[n]_q}\biggr)\\
&=C(m,n,q,x)\cdot \frac{1}{[n]_q}
\end{align*}
applied to $m=m_r$, $n=n_1$, $a=m_{r+1}$ 
in the definition of $Z_N^{\star}(k_1, \dots, k_r+1,0;l_1, \dots, l_s;q;x)$.
\end{proof}

\begin{corollary}
Let $N$ be a positive integer and $\bb{k}=(k_1, \dots, k_r)$ an index. 
We define $P_N(\bb{k};q;x)$, $Q_N(\bb{k};q;x)$ and $R_N(q;x)$ by 
\begin{align*}
P_N(\bb{k};q;x)&:=\sum_{1\leq m_1\leq\cdots\leq m_r\leq N}
\frac{[m_1]_q}{[m_1]_q-q^{m_1}x}
\prod_{i=1}^r\frac{q^{(k_i-1)m_i}}{[m_i]_q([m_i]_q-q^{m_i}x)^{k_i-1}}
\cdot\frac{[m_r]_q}{[m_r]_q-q^{m_r}x}\\ 
&\hspace{7em}\cdot(-1)^{m_r-1}q^{\frac{m_r(m_r+1)}{2}}\binom{N}{m_r}_q,\\
Q_N(\bb{k};q;x)&:=\sum_{1\leq m_1\leq\cdots\leq m_r\leq N}
\prod_{i=1}^r\frac{q^{m_i}}{([m_i]_q-q^{m_i}x)[m_i]_q^{k_i-1}},\\
R_N(q;x)&:=\prod_{h=1}^N\biggl(1-\frac{q^hx}{[h]_q}\biggr)^{-1}.
\end{align*}
Then we have
\begin{equation}\label{P=QR}
P_N(\bb{k};q;x)=Q_N(\bb{k}^{\vee};q;x)R_N(q;x).
\end{equation}
\end{corollary}
\begin{proof}
By applying equalities in Theorem \ref{step-thm} $\wt(\bb{k})$ times, 
we see that
\[
Z_N^{\star}(\bb{k};\varnothing;q;x)=\cdots =Z_N^{\star}(0;\bb{k}^{\vee};q;x)
\]
holds by the definition of the Hoffman dual. For example, 
\[
Z_N^{\star}(1,1,2;\varnothing)
\stackrel{\eqref{+1;1,}}{=}Z_N^{\star}(1,1,1;1)\stackrel{\eqref{+1;1,}}{=}Z_N^{\star}(1,1,0;1,1)
\stackrel{\eqref{+1,0;1+}}{=}Z_N^{\star}(1,0;2,1)
\stackrel{\eqref{+1,0;1+}}{=}Z_N^{\star}(0;3,1) 
\]
(here we abbreviated $Z_N^{\star}(\bb{k};\bb{l};q;x)$ as $Z_N^{\star}(\bb{k};\bb{l})$). 
By definition, we have
\begin{align*}
Z_N^{\star}(\bb{k};\varnothing;q;x)
&=\sum_{1\leq m_1\leq\cdots\leq m_r\leq N}F_1(\bb{k};\bb{m};q;x)C(m_r,N,q,x)\\
&=P_N(\bb{k};q;x)R_N(q;x)^{-1}
\end{align*}
and
\begin{align*}
Z_N^{\star}(0;\bb{k}^{\vee};q;x)
&=\sum_{1\leq m\leq n_1\leq\cdots\leq n_s\leq N}
\frac{q^{-m}[m]_q}{[m]_q-q^mx}C(m,n_1,q,x)F_2(\bb{k}^\vee;\bb{n};q;x)\\
&=Q_N(\bb{k}^{\vee};q;x).
\end{align*}
In the last equality, we have used the partial fraction decomposition
\[
\sum_{m=1}^{n_1}\frac{[m]_q}{[m]_q-q^mx}
\cdot\frac{(-1)^{m-1}q^{\frac{m(m-1)}{2}}}{[m]_q![n_1-m]_q!}
=\frac{1}{\prod_{h=1}^{n_1}([h]_q-q^hx)}.
\]
The proof is complete. 
\end{proof}

\begin{proof}[Proof of Theorem $\ref{MT1}$]
By using the expansion formula
\[
\frac{1}{([m]_q-q^mx)^k}
=\sum_{e=0}^{\infty}\binom{k+e-1}{e}\frac{q^{em}x^e}{[m]_q^{k+e}}
\]
for a positive integer $m$ and a non-negative integer $k$, we see that 
\[
P_N(\bb{k};q;x)=\sum_{e=0}^{\infty}\sum_{\bb{e} \in J_{e,r}}
b(\bb{k};\bb{e})H_N^{\star}(\bb{k}+\bb{e};q)x^e
\]
and
\[
Q_N(\bb{k}^{\vee};q;x)=\sum_{e=0}^{\infty}\sum_{\bb{e} \in J_{e,s}}
z_N^{\star}(\bb{k}^{\vee};\bb{e};q)x^e.
\]
Since $R_N(q;x)=\sum_{e=0}^{\infty}z_N^{\star}(\{1\}^{e};q)x^e$, we obtain 
the identity \eqref{q-Ohno} by comparing the coefficients of $x^e$ in \eqref{P=QR}.
\end{proof}

\section{Sum formulas for $\calA_2$-finite multiple zeta values}\label{sec:sum formula}
\subsection{Auxiliary facts}
We prepare some known facts for finite multiple zeta values. 

\begin{proposition}[{\cite[Theorem 6.1, 6.2]{H}, \cite[Theorem 3.1, 3.5]{Z1}}]\label{Aux-1}
Let $k_1, k_2$ and $k_3$ be positive integers, and assume that $l:=k_1+k_2+k_3$ is odd. 
Then 
\begin{align}
\label{A-2}
\zeta_{\calA}^{\star}(k_1,k_2)
&=(-1)^{k_2}\binom{k_1+k_2}{k_1}\frac{B_{\bb{p}-k_1-k_2}}{k_1+k_2}, \\
\label{A-3}
\zeta_{\calA}^{\star}(k_1,k_2,k_3)
&=\frac{1}{2}\biggl\{(-1)^{k_3}\binom{l}{k_3}-(-1)^{k_1}\binom{l}{k_1}\biggr\}
\frac{B_{\bb{p}-l}}{l}.
\end{align}
\end{proposition}

\begin{proposition}[{\cite{ZC}, \cite[Theorem 3.2]{Z1}}]\label{Aux-2}
Let $k, r, k_1$ and $k_2$ be positive integers, and assume that $l:=k_1+k_2$ is even. 
Then 
\begin{align}
\label{A_2-{k}^r}
\zeta_{\calA_2}^{\star}(\{k\}^r )
&=k\frac{B_{\boldsymbol{p}-rk-1}}{rk+1}\bb{p},\\
\label{A_2-2}
\zeta_{\calA_2}^{\star}(k_1,k_2)
&=\frac{1}{2}\biggl\{(-1)^{k_1}k_2\binom{l+1}{k_1}-(-1)^{k_2}k_1\binom{l+1}{k_2}+l\biggr\}
\frac{B_{\bb{p}-l-1}}{l+1}\bb{p}.
\end{align}
\end{proposition}

\begin{proposition}[{\cite[Corollary 3.16 (42)]{SS}}]
Let $n$ be a positive integer and $\bb{k}=(k_1, \dots, k_r)$ an index. Then 
\begin{equation}\label{non-star to star}
\sum_{j=0}^r (-1)^j\zeta_{\calA_n}(k_j, \dots, k_1)\,
\zeta_{\calA_n}^{\star}(k_{j+1}, \dots, k_r)=0.
\end{equation}
\end{proposition}

\subsection{Computations of sums for $\mathcal{A}_2$-finite multiple zeta values}
\begin{definition}
Let $k, r$ and $i$ be positive integers satisfying $i\leq r\leq k$. 
We define four sums $S_{k,r}$, $S_{k,r}^{\star}$, 
$S_{k,r,i}$ and $S_{k,r,i}^{\star}$ in $\calA_2$ by
\begin{alignat*}{2}
S_{k,r}&:=\sum_{\bb{k} \in I_{k,r}}\zeta_{\calA_2}(\bb{k}), &\qquad 
S_{k,r}^{\star}&:=\sum_{\bb{k} \in I_{k,r}}\zeta_{\calA_2}^{\star}(\bb{k}),\\
S_{k,r,i}&:=\sum_{\bb{k} \in I_{k,r,i}}\zeta_{\calA_2}(\bb{k}), &\qquad 
S_{k,r,i}^{\star}&:=\sum_{\bb{k} \in I_{k,r,i}}\zeta_{\calA_2}^{\star}(\bb{k}).
\end{alignat*}
\end{definition}

For an index $\bb{k}=(k_1, \dots, k_r)$, we set $\bb{k}^+:=(k_1, \dots, k_{r-1}, k_r+1)$.
We can calculate $S_{k,r}^{\star}$ and $S_{k,r,i}^{\star}$ by using the following identity. 

\begin{corollary}
Let $e$ be a non-negative integer, $\bb{k} \in I_{\ast,r}$ an index 
and $s:=\dep(\bb{k}^{\vee})$. Then we have
\begin{equation}\label{A_2-Ohno}
\begin{split}
&\sum_{j=0}^e\zeta_{\calA_2}^{\star}(\{1\}^{e-j})
\sum_{\bb{e}\in J_{j,r}}\zeta_{\calA_2}^{\star}(\bb{k}+\bb{e})\\
&=\sum_{\bb{e}'\in J_{e,s}}b(\bb{k}^{\vee};\bb{e}')
\Bigl\{-\zeta_{\calA_2}^{\star}(\bb{k}^{\vee}+\bb{e}')
-\zeta_{\calA_2}^{\star}(\bb{k}^{\vee}+\bb{e}', 1)\bb{p}
+\zeta_{\calA_2}^{\star}((\bb{k}^{\vee}+\bb{e}')^+)\bb{p}\Bigr\}. 
\end{split}
\end{equation}
\end{corollary}
\begin{proof}
Since a congruence
\[
(-1)^{m-1}\binom{p-1}{m}
\equiv -1-\sum_{m\leq n\leq p-1}\frac{p}{n}+\frac{p}{m} \pmod{p^2}
\]
holds for any odd prime $p$ and any positive integer $m$ with $m<p$ 
(cf. \cite[Lemma 4.1]{S}), this corollary is a direct consequence of \eqref{Ohno}.
\end{proof}

\begin{proposition}\label{S_A_2,k,r}
For positive integers $k$ and $r$ such that $r\leq k$, we have 
\[
(-1)^{r-1}S_{k,r}=S_{k,r}^{\star}=\binom{k}{r}\frac{B_{\bb{p}-k-1}}{k+1}\bb{p}.
\]
\end{proposition}
\begin{proof}
Let $\bb{k}=(\{1\}^r)$ and $e=k-r$ in \eqref{A_2-Ohno}. 
Then $\bb{k}^{\vee}=(r)$ and we have
\begin{equation}\label{A_2-app1}
\sum_{j=0}^{k-r}\zeta_{\calA_2}^{\star}(\{1\}^{k-r-j})S_{j+r,r}^{\star}
=\binom{k}{r}\Bigl\{-\zeta_{\calA_2}(k)-\zeta_{\calA_2}^{\star}(k,1)\bb{p}
+\zeta_{\calA_2}(k+1)\bb{p}\Bigr\}.
\end{equation}
For $0\leq j<k-r$, $\zeta_{\calA_2}^{\star}(\{1\}^{k-r-j})S_{j+r,r}^{\star}=0$ 
since both $\zeta_{\calA_2}^{\star}(\{1\}^{k-r-j})$ and $S_{j+r,r}^{\star}$ 
are divisible by $\bb{p}$ by \eqref{A_2-{k}^r} and \eqref{eq:k,r sum A}. 
Therefore, the left hand side of \eqref{A_2-app1} is equal to $S_{k,r}^{\star}$. 
On the other hand, the right hand side of \eqref{A_2-app1} is equal to
\[
\binom{k}{r}\biggl\{-k\frac{B_{\bb{p}-k-1}}{k+1}\bb{p}
+\binom{k+1}{k}\frac{B_{\bb{p}-k-1}}{k+1}\bb{p}\biggr\}
=\binom{k}{r}\frac{B_{\bb{p}-k-1}}{k+1}\bb{p}
\]
by \eqref{A_2-{k}^r} and \eqref{A-2}. 
Hence we have $S_{k,r}^{\star}=\binom{k}{r}\frac{B_{\bb{p}-k-1}}{k+1}\bb{p}$. 

By taking $\sum_{\bb{k} \in I_{k,r}}$ of \eqref{non-star to star}, we obtain
\[
S_{k,r}^{\star}+\sum_{j=1}^{r-1}(-1)^j\sum_{l=j}^{k-r+j}S_{l,j}S_{k-l,r-j}^{\star}+(-1)^rS_{k,r}=0.
\]
We see that $S_{l,j}S_{k-l,r-j}^{\star}=0$ for $1\leq j\leq r-1$ and $j \leq l \leq k-r+j$, 
since both $S_{l,j}$ and $S_{k-l,r-j}^{\star}$ are divisible by $\bb{p}$ 
by \eqref{eq:k,r sum A}. This gives $(-1)^{r-1}S_{k,r}=S_{k,r}^{\star}$. 
\end{proof}

Next we compute $S_{k,r,i}^{\star}$ and $S_{k,r,i}$. 

%\begin{lemma}\label{PFD}
%For a non-negative integer $j$ and positive integers $a, b$, 
%\[
%D_j(a,b):=\sum_{\substack{e_1+e_2=j \\ e_1, e_2 \geq 0}}
%\frac{(-1)^{e_1}}{e_1!e_2!(a+e_1)(b+e_2)}
%=\frac{1}{a+b+j}\biggl\{\frac{(a-1)!}{(a+j)!}+(-1)^j\frac{(b-1)!}{(b+j)!}\biggr\}.
%\]
%\end{lemma}

%\begin{proof}
%Since
%\[
%D_j(a,b)=\frac{1}{a+b+j} 
%\Biggl\{\frac{1}{j!}\sum_{e_1=0}^j(-1)^{e_1}\binom{j}{e_1}\frac{1}{a+e_1}
%+(-1)^j\frac{1}{j!}\sum_{e_2=0}^j(-1)^{e_2}\binom{j}{e_2}\frac{1}{b+e_2}\Biggr\},
%\]
%this follows from the partial fraction decomposition 
%\[\frac{1}{x(x+1)\cdots(x+j)}
%=\sum_{e=0}^j\frac{(-1)^e}{e!\,(j-e)!}\cdot\frac{1}{x+e}
%=\frac{1}{j!}\sum_{e=0}^j(-1)^e\binom{j}{e}\frac{1}{x+e}. \qedhere\]
%\end{proof}

\begin{theorem}\label{thm:k,r,i sum A_2}
Let $k, r$ and $i$ be positive integers satisfying $i\leq r<k$, and 
assume that $k$ is even. Then we have
\[
S_{k,r,i}=(-1)^{r-1}\frac{a_{k,r,i}}{2}\cdot\frac{B_{\bb{p}-k-1}}{k+1}\bb{p}, \quad 
S_{k,r,i}^{\star}=\frac{b_{k,r,i}}{2}\cdot\frac{B_{\bb{p}-k-1}}{k+1}\bb{p},
\]
where
\begin{align*}
a_{k,r,i}&=\binom{k-1}{r}
+(-1)^{r-i}\biggl\{(k-r)\binom{k}{i-1}+\binom{k-1}{i-1}+(-1)^{r-1}\binom{k-1}{r-i}\biggr\},\\
b_{k,r,i}&=\binom{k-1}{r}
+(-1)^{i-1}\biggl\{(k-r)\binom{k}{r-i}+\binom{k-1}{r-i}+(-1)^{r-1}\binom{k-1}{i-1}\biggr\}.
\end{align*}
\end{theorem}

\begin{proof}
Let $\bb{k}=(\{1\}^{i-1},2,\{1\}^{r-i})$ and $e=k-r-1$ in \eqref{A_2-Ohno}. 
Then $\bb{k}^{\vee}=(i,r-i+1)$ and we have
\begin{equation}\label{A_2-app2}
\begin{split}
&\sum_{j=0}^{k-r-1}\zeta_{\calA_2}^{\star}(\{1\}^{k-r-1-j})S_{j+r+1,r,i}^{\star}\\
&=\sum_{e=0}^{k-r-1}
\binom{i+e-1}{e}\binom{k-i-e-1}{k-r-1-e}
\Bigl\{-\zeta_{\calA_2}^{\star}(i+e,k-i-e)\\
&\hspace{5em}-\zeta_{\calA_2}^{\star}(i+e,k-i-e,1)\bb{p}
+\zeta_{\calA_2}^{\star}(i+e,k-i-e+1)\bb{p}\Bigr\}.
\end{split}
\end{equation}
For $0\leq j<k-r-1$, we see that 
$\zeta_{\calA_2}^{\star}(\{1\}^{k-r-1-j})S_{j+r+1,r,i}^{\star}$ is a rational multiple 
of $B_{\bb{p}-k+r+j}B_{\bb{p}-j-r-1}\bb{p}$ by \eqref{A_2-{k}^r} and Theorem \ref{SW-thm}. 
Since $k$ is even, one of $B_{\bb{p}-k+r+j}$ or $B_{\bb{p}-j-r-1}$ is zero. 
Therefore, the left hand side of \eqref{A_2-app2} is equal to $S_{k,r,i}^{\star}$. 

On the other hand, we can calculate the right hand side of \eqref{A_2-app2} as follows. 
By \eqref{A_2-2}, \eqref{A-2} and \eqref{A-3}, we have 
\begin{align*}
&-\zeta_{\calA_2}^{\star}(i+e,k-i-e)
-\zeta_{\calA_2}^{\star}(i+e,k-i-e,1)\bb{p}
+\zeta_{\calA_2}^{\star}(i+e,k-i-e+1)\bb{p}\\
&=\Biggl[-\frac{1}{2}\biggl\{(-1)^{i+e}(k-i-e)\binom{k+1}{i+e}
-(-1)^{k-i-e}(i+e)\binom{k+1}{k-i-e}+k\biggr\}\\
&\qquad -\frac{1}{2}\biggl\{-(k+1)-(-1)^{i+e}\binom{k+1}{i+e}\biggr\}
+(-1)^{k-i-e+1}\binom{k+1}{i+e}\Biggr]\frac{B_{\bb{p}-k-1}}{k+1}\bb{p}\\
&=\frac{1}{2}\Biggl[1-(-1)^{i+e}(k-i-e+1)\binom{k+1}{i+e}
+(-1)^{i+e}(i+e)\binom{k+1}{k-i-e}\Biggr]\frac{B_{\bb{p}-k-1}}{k+1}\bb{p}\\
&=\frac{1}{2}\Biggl[1+(-1)^{i-1+e}\binom{k+1}{i+e+1}\Biggr]\frac{B_{\bb{p}-k-1}}{k+1}\bb{p}. 
\end{align*}
Therefore, the right hand side of \eqref{A_2-app2} is equal to 
\begin{equation}\label{A_2-int}
\frac{1}{2}\sum_{e=0}^{k-r-1}
\binom{i+e-1}{e}\binom{k-i-e-1}{k-r-1-e}
\Biggl[1+(-1)^{i-1+e}\binom{k+1}{i+e+1}\Biggr]\frac{B_{\bb{p}-k-1}}{k+1}\bb{p}.
\end{equation}
By comparing the coefficient of $x^{k-r-1}$ in $(1-x)^{-i}(1-x)^{-(r-i+1)}=(1-x)^{-(r+1)}$, we see that
\[
\sum_{e=0}^{k-r-1}\binom{i+e-1}{e}\binom{k-i-e-1}{k-r-1-e}=\binom{k-1}{r},
\]
and by using the partial fraction decomposition 
\[
F(x):=\sum_{e=0}^{k-r-1}\frac{(-1)^e}{e!(k-r-1-e)!}\cdot\frac{1}{x+e}=\frac{1}{x(x+1)\cdots (x+k-r-1)},
\]
we see that
\begin{align*}
&\sum_{e=0}^{k-r-1}\binom{i+e-1}{e}\binom{k-i-e-1}{k-r-1-e}\cdot(-1)^{i-1+e}\binom{k+1}{i+e+1} \\
&=(-1)^{i-1}\frac{(k+1)!}{(i-1)!(r-i)!}\sum_{e=0}^{k-r-1}\frac{(-1)^e}{e!(k-r-1-e)!(i+e)(i+e+1)(k-i-e)}\\
&=(-1)^{i-1}\frac{(k+1)!}{(i-1)!(r-i)!}\left\{\frac{1}{k}F(i)-\frac{1}{k+1}F(i+1)+\frac{(-1)^{r-1}}{k(k+1)}F(r-i+1)\right\}\\
&=(-1)^{i-1}\biggl\{(k-r)\binom{k}{r-i}+\binom{k-1}{r-i}+(-1)^{r-1}\binom{k-1}{i-1}\biggr\}.
\end{align*}
Thus we have proved $S_{k,r,i}^{\star}=\frac{b_{k,r,i}}{2}\cdot\frac{B_{\bb{p}-k-1}}{k+1}\bb{p}$. 

Let us take the sum $\sum_{\bb{k} \in I_{k,r,r+1-i}}$ of \eqref{non-star to star}. 
Then we obtain
\begin{multline}\label{A_2-SS*}
S_{k,r,r+1-i}^{\star}
+\sum_{j=1}^{r-i}(-1)^j\sum_{l=j}^{k-r+j-1}S_{l,j}S_{k-l,r-j,r+1-i-j}^{\star}\\ 
+\sum_{j=r-i+1}^{r-1}(-1)^j\sum_{l=j+1}^{k-r+j}S_{l,j,j+i-r}S_{k-l,r-j}^{\star}
+(-1)^rS_{k,r,i}=0.
\end{multline}
We know that 
$S_{l,j}S_{k-l,r-j,r+1-i-j}^\star$ is a rational multiple of 
$B_{\bb{p}-l-1}B_{\bb{p}-k+l}\bb{p}$ for $1\leq j\leq r-i$ and
we also know that 
$S_{l,j,j+i-r}S_{k-l,r-j}^{\star}$ is a rational multiple of 
$B_{\bb{p}-l}B_{\bb{p}-k+l-1}\bb{p}$ for $r-i+1\leq j\leq r-1$ 
by Theorem \ref{SW-thm} and Proposition \ref{S_A_2,k,r}. 
Since $k$ is even, these are zero for every $l$. 
Therefore, we have
\[
S_{k,r,i}=(-1)^{r-1}\frac{b_{k,r,r+1-i}}{2}\cdot\frac{B_{\bb{p}-k-1}}{k+1}\bb{p}
=(-1)^{r-1}\frac{a_{k,r,i}}{2}\cdot\frac{B_{\bb{p}-k-1}}{k+1}\bb{p}. \qedhere
\]
\end{proof}

\section{Sum formulas for $\calA_3$-finite multiple zeta values}
\label{sec:A_3-sum formula}
For positive integers $k$ and $r$ such that $r\leq k$, we set 
\[
T_{k,r}:=\sum_{\bb{k}\in I_{k,r}}\zeta_{\calA_3}(\bb{k}),\quad 
T_{k,r}^{\star}:=\sum_{\bb{k}\in I_{k,r}}\zeta_{\calA_3}^{\star}(\bb{k}).
\]
From now on, we assume that $k$ is odd. We recall a formula
\begin{equation}\label{A_3-single}
\zeta_{\calA_3}(k)=-\frac{k(k+1)}{2}\cdot\frac{B_{\bb{p}-k-2}}{k+2}\bb{p}^2
\end{equation}
proved by Sun \cite[Theorem 5.1]{ZHS}.
\begin{theorem}\label{thm:A_3-sum}
Let $k$ and $r$ be positive integers satisfying $r\leq k$, and 
assume that $k$ is odd. Then we have 
\[
(-1)^{r-1}T_{k,r}=T_{k,r}^{\star}=-\frac{k+1}{2}\binom{k}{r}\frac{B_{\bb{p}-k-2}}{k+2}\bb{p}^2.
\]
\end{theorem}
\begin{proof}
Since a congruence
\begin{align*}
&(-1)^{m-1}\binom{p-1}{m}\\
&\equiv -1-\left(\sum_{m\leq n\leq p-1}\frac{1}{n}-\frac{1}{m}\right)p-\left(\sum_{m\leq n_1\leq n_2\leq p-1}\frac{1}{n_1n_2}-\frac{1}{m}\sum_{m\leq n\leq p-1}\frac{1}{n}\right)p^2 \pmod{p^3}
\end{align*}
holds for any odd prime $p$ and any positive integer $m$ with $m<p$ 
by \cite[Lemma 4.1]{S}, one can deduce 
\begin{equation}\label{A_3-app}
\begin{split}
&\sum_{j=0}^{k-r}\zeta_{\calA_3}^{\star}(\{1\}^{k-r-j})T_{j+r,r}^{\star}\\
&=\binom{k}{r}\left\{-\zeta_{\calA_3}(k)-\zeta_{\calA_3}^{\star}(k,1)\bb{p}
+\zeta_{\calA_3}^{\star}(k+1)\bb{p}-\zeta_{\calA_3}^{\star}(k,1,1)\bb{p}^2
+\zeta_{\calA_3}^{\star}(k+1,1)\bb{p}^2\right\}
\end{split}
\end{equation}
from the identity \eqref{Ohno} in the samy way as \eqref{A_2-app1}. Let us fix $0 \leq j < k-r$. By \eqref{A_2-{k}^r} and Proposition \ref{S_A_2,k,r}, if $j+r$ is odd, then $\zeta_{\calA_3}^{\star}(\{1\}^{k-r-j})$ is divisible by $\bb{p}$ and $T_{j+r,r}^{\star}$ is divisible by $\bb{p}^2$ and if $j+r$ is even, then $\zeta_{\calA_3}^{\star}(\{1\}^{k-r-j})$ is divisible by $\bb{p}^2$ and $T_{j+r,r}^{\star}$ is divisible by $\bb{p}$. Therefore, $\zeta_{\calA_3}^{\star}(\{1\}^{k-r-j})T_{j+r,r}^{\star}=0$ and we see that the left hand side of \eqref{A_3-app} is equal to $T_{k,r}^{\star}$. On the other hand, by using Proposition \ref{Aux-1}, Proposition \ref{Aux-2} and \eqref{A_3-single}, we see that the right hand side of \eqref{A_3-app} is equal to 
\begin{align*}
&\binom{k}{r}\Biggl[\frac{k(k+1)}{2}-\frac{1}{2}\left\{-\binom{k+2}{k}+k^2+3k+1\right\}+(k+1)\\
&\qquad-\frac{1}{2}\left\{-(k+2)+\binom{k+2}{k}\right\}-(k+2)\Biggr]\frac{B_{\bb{p}-k-2}}{k+2}\bb{p}^2\\
&=-\frac{k+1}{2}\binom{k}{r}\frac{B_{\bb{p}-k-2}}{k+2}\bb{p}^2.
\end{align*}
Hence we have $T_{k,r}^{\star}=-\frac{k+1}{2}\binom{k}{r}\frac{B_{\bb{p}-k-2}}{k+2}\bb{p}^2$. By taking $\sum_{\bb{k} \in I_{k,r}}$ of \eqref{non-star to star}, we obtain
\[
T_{k,r}^{\star}+\sum_{j=1}^{r-1}(-1)^j\sum_{l=j}^{k-r+j}T_{l,j}T_{k-l,r-j}^{\star}+(-1)^rT_{k,r}=0.
\]
Let us fix $1\leq j\leq r-1$ and $j \leq l \leq k-r+j$. By Proposition \ref{S_A_2,k,r}, if $l$ is odd, then $T_{l,j}$ is divisible by $\bb{p}^2$ and $T_{k-l,r-j}^{\star}$ is divisible by $\bb{p}$ and if $l$ is even, then $T_{l,j}$ is divisible by $\bb{p}$ and $T_{k-l,r-j}^{\star}$ is divisible by $\bb{p}^2$. Therefore, we see that $T_{l,j}T_{k-l,r-j}^{\star}=0$ and this gives $(-1)^{r-1}T_{k,r}=T_{k,r}^{\star}$. 
\end{proof}
%%%%%%%%%%%%%%%%%%%%%%%%%%%%%%%%%%%%
	
\end{document}